\documentclass[reqno,11pt]{amsart}
\numberwithin{equation}{section}

\usepackage{color,graphics,epsfig}

\newcommand{\calA}{\mathcal{A}}

\newcommand{\calG}{\mathcal{G}}

\newcommand{\calI}{\mathcal{I}}

\newcommand{\calL}{\mathcal{L}}

\newcommand{\mA}{\mathbb{A}}

\newcommand{\mC}{\mathbb{C}}
\newcommand{\mD}{\mathbb{D}}

\newcommand{\mF}{\mathbb{F}}

\newcommand{\mN}{\mathbb{N}}

\newcommand{\mR}{\mathbb{R}}
\newcommand{\mS}{\mathbb{S}}
\newcommand{\mT}{\mathbb{T}}

\newcommand{\mZ}{\mathbb{Z}}

\newcommand{\inv}{{\textrm{inv }}}

\newtheorem{theorem}{Theorem}[section]

\newtheorem{proposition}[theorem]{Proposition}

\theoremstyle{definition}

\newtheorem{remark}[theorem]{Remark}

\theoremstyle{definition}
\newtheorem{definition}[theorem]{Definition}

\theoremstyle{definition}

\begin{document}



\title[On two natural extensions of the $\nu$-metric]{On two natural extensions of Vinnicombe's metric: their noncoincidence 
yet equivalence on stabilizable plants over $\calA_+$}

\author{Rudolf Rupp}
\thanks{Part of this research was done while the first author enjoyed a sabbatical. He wants to thank the University of Applied Science N\"urnberg for this.}
\address{Fakult\"at Allgemeinwissenschaften,
  Georg-Simon-Ohm Hochschule N\"urnberg,
  Ke\ss lerplatz 12, D-90489 N\"urnberg,
  Germany.}
\email{rudolf.rupp@ohm-hochschule.de}

\author{Amol Sasane}
\address{Department of Mathematics, London School of Economics,
    Houghton Street, London WC2A 2AE, United Kingdom.}
\email{a.j.sasane@lse.ac.uk}

\begin{abstract}
  Let  $\calA_+$ be the ring of Laplace transforms of complex Borel measures on $\mR$ 
  with support in $[0,+\infty)$ which do not have a singular nonatomic part. We compare 
  the $\nu$-metric $d_{\calA_+}$ for stabilizable plants over $\calA_+$ given in \cite{BalSas} 
  with yet another metric $d_{H^\infty}|_{\calA_+}$, namely the one  induced by the metric $d_{H^\infty}$ for the set of stabilizable plants over $H^\infty$ given in 
  \cite{Sas}. Both $d_{\calA_+}$ and $d_{H^\infty}$ coincide with the classical 
Vinnicombe metric defined for rational transfer functions, but we show here 
by means of an example that these two possible extensions of the 
  classical $\nu$-metric for plants over $\calA_+$ do not coincide on the set of stabilizable plants over $\calA_+$. We also prove that they nevertheless 
give rise to the same topology on stabilizable plants over $\calA_+$, which in turn coincides with the gap metric topology. 
\end{abstract}

\maketitle

\section{Introduction}

We recall the general {\em stabilization problem} in control theory.
Suppose that $R$ is a commutative integral domain with identity
(thought of as the class of stable transfer functions) and let
$\mF(R)$ denote the field of fractions of $R$. The stabilization
problem is: Given $P\in \mF(R)$ (an unstable plant transfer function), 
find $C \in \mF(R)$ (a stabilizing controller
  transfer function), such that 
$$
H(P,C):= \left[\begin{array}{cc} P \\ 1 \end{array} \right]
(1-CP)^{-1} \left[\begin{array}{cc} -C & 1 \end{array} \right]
 \in R^{2 \times 2} \textrm{ (is stable).}
$$
In the {\em robust stabilization problem}, one goes a step
further.  One knows that the plant is just an approximation of
reality, and so one would really like the controller $C$ to not only
stabilize the {\em nominal} plant $P_0$, but also all sufficiently
close plants $P$ to $P_0$.  The question of what one means by
``closeness'' of plants thus arises naturally. So one needs a  
function $d$ defined on pairs of stabilizable plants such that 
$d$ is a metric on the set of all stabilizable plants, $d$ is amenable to computation, and 
stabilizability is a robust property of the plant with respect
  to this metric (that is, whenever a plant $P_0$ is stabilized by 
  a controller $C$, then there is a small enough neighbourhood of the plant $P_0$  
  consisting of plants which are stabilized by the same controller $C$).  
Such a desirable metric was introduced by Glenn Vinnicombe in
\cite{Vin} and is called the $\nu$-{\em metric}. In that paper,
essentially $R$ was taken to be the rational functions without poles
in the closed unit disk. It was shown in \cite{Vin} that the $\nu$-metric is indeed a metric
on the set of stabilizable plants, and that  stabilizability 
is a robust property of the plant $P$.

The problem of what happens when $R$ is some other ring of stable
transfer functions of infinite-dimensional systems was left open in
\cite{Vin}. This problem of extending the $\nu$-metric from the
rational case to nonrational transfer function classes of infinite-dimensional
systems was addressed in \cite{BalSas} where the approach taken was abstract. 
However when we focus on the set of stabilizable plants over $\calA_+$, 
there are two possible natural extensions of the Vinnicombe metric for rational plants. 
We recall these two possibilities from \cite{BalSas} and \cite{Sas} in the following section. 
The question of whether these two metrics coincide on the full set of stabilizable plants over $\calA_+$ is a natural one. We prove that 
this is not the case by means of an example in Section~\ref{section_example}. Notwithstanding this noncoincidence, we show that 
these two metrics do induce the same topology in Section~\ref{section_equivalence}.

\section{Recap of the two $\nu$-metrics for unstable plants over $\calA_+$}
\label{section_abstract_nu_metric}

\noindent We recall the following standard definitions from the
factorization approach to control theory. 

\subsection{The notation $\mF(\calA_+)$:}\label{subsec1}  
$\mF(\calA_+)$ denotes the field of fractions of $\calA_+$.

\subsection{Normalized coprime factorization:} For 
 a $P \in \mF(\calA_+)$, a factorization $P=N/D$,
where $N,D\in \calA_+$, is called a {\em coprime 
factorization of} $P$ if there exist $X, Y\in 
\calA_+$ such that $ X N + Y D=1$.  If moreover 
$$ 
\overline{N(iy)} N(iy) +\overline{D(iy)}D(iy) =1 \quad (y\in \mR),
$$
then the coprime factorization is
referred to as a {\em normalized} coprime factorization of $P$. 
 Since we are dealing with functions rather than with matrices, it is not necessary to distinguish between left and right coprime factorizations.

\subsection{The notation $G, \widetilde{G}$, $K$, $\widetilde{K}$:}
\label{subsec5}
Given $P \in \mF(\calA_+)$ with normalized factorization $P=N /D$, 
we introduce the following matrices with entries from
$\calA_+$:
$$
G=\left[ \begin{array}{cc} N \\ D \end{array} \right] \quad
\textrm{and} \quad  
\widetilde{G}=\left[ \begin{array}{cc} -D & N \end{array} \right] .
$$
Similarly, given an element $C \in \mF(\calA_+)$ with normalized coprime
factorization $C=X/Y$, we introduce the following matrices with
entries from $\calA_+$:
$$
K=\left[ \begin{array}{cc} Y \\ X \end{array} \right] \quad
\textrm{and} \quad \widetilde{K}=\left[ \begin{array}{cc} -X & Y\end{array} \right] .
$$

\subsection{The notation $\mS(\calA_+)$:}\label{subsec6} 
$\mS(\calA_+)$ denotes the set of all $P\in \mF(\calA_+)$ 
that possess a normalized coprime factorization. 

It follows from the proof of \cite[Lemma~6.5.6.(e)]{Mik} and
\cite[Theorem~5.2.8]{Mik} that whenever $p\in \mF(\calA_+)$ has a
coprime factorization over $\calA_+$, it also has a {\em normalized}
coprime factorization over $\calA_+$. However, it is known that not 
every element in $\mF(\calA_+)$ possesses a coprime factorization; see for example \cite{Log}.

We now recall the definition of the two metrics $d_\nu$ on $\mS(\calA_+)$.

\subsection{The metric $d_{\calA_+}$}

Let $\mC_{\geq 0}:=\{s\in \mC\;|\; \textrm{Re}(s)\geq 0\}$ and let
$\calA^+$ denote the Banach algebra

$
\!\!\!\!\calA^+=\left\{ s (\in \mC_{\geq 0}) \mapsto \widehat{f_a}(s)
  +\displaystyle \sum_{k=0}^\infty f_k e^{- s t_k} \; \bigg| \;
\begin{array}{ll}
f_a \in L^{1}(0,\infty), \;(f_k)_{k\geq 0} \in \ell^{1},\\
0=t_0 <t_1 ,t_2 , t_3, \dots
\end{array} \right\}
$

\noindent equipped with pointwise operations and the norm:

$
\|F\|=\|f_a\|_{\scriptscriptstyle L^{1}} +
\|(f_k)_{k\geq 0}\|_{\scriptscriptstyle \ell^1},
\;\; F(s)=\widehat{f_a}(s) +\displaystyle\sum_{k=0}^\infty f_k e^{-st_k}\;\;(s\in \mC_{\geq 0}).
$

\noindent Here $\widehat{f_a}$ denotes the {\em Laplace transform of}
$f_a$, given by
$$
\widehat{f_a}(s)=\displaystyle \int_0^\infty e^{-st} f_a(t)
dt, \quad s \in \mC_{\geq 0}.
$$
Similarly, define the Banach algebra $\calA$ as follows:

$
\!\!\!\!\!\!\calA\!=\!\left\{ iy (\in i\mR) \mapsto \widehat{f_a}(iy)
  +\!\!\!\displaystyle\sum_{k=-\infty}^\infty f_k e^{- iy t_k} \bigg|
\begin{array}{ll}
f_a \in L^{1}(\mR), \;(f_k)_{k\in \mZ } \in \ell^{1},\\
\dots, t_{-2}, t_{-1}<\!0\!=\!t_0\! <t_1 ,t_2 ,  \dots
\end{array} \!\!\!\right\}
$

\noindent equipped with pointwise operations and the norm:

$
\|F\|=\|f_a\|_{\scriptscriptstyle L^{1}} + \|(f_k)_{k\in
  \mZ}\|_{\scriptscriptstyle \ell^1}, \;\; F(iy):=\widehat{f_a}(iy)
+\displaystyle\sum_{k=-\infty}^\infty f_k e^{-iy t_k}\;\;(y\in \mR).
$

\noindent Here $\widehat{f_a}$ is the {\em Fourier transform of}
$f_a$, 
$$
\widehat{f_a}(iy)= \displaystyle \int_{-\infty}^\infty e^{-iyt} f_a(t)
dt \quad (y \in \mR).
$$
For $ F(iy)=\widehat{f_a}(iy) +\displaystyle\sum_{k=-\infty}^\infty
f_k e^{-iy t_k}$ $(y\in \mR)$ in $\calA$, we set
$$
F_{AP}(iy)=\displaystyle\sum_{k=-\infty}^\infty f_k e^{-iy t_k}\quad 
(y\in \mR),
$$
and call it the {\em almost periodic part of $F$}. 

Recall that the algebra $AP$ of complex valued (uniformly) {\em almost periodic
  functions} is the smallest closed subalgebra of $L^\infty(\mR)$ that
contains all the functions $e_\lambda := e^{i \lambda y}$. Here the
parameter $\lambda$ belongs to $\mR$.  For any $f\in AP$, its {\em
  Bohr-Fourier series} is defined by the formal sum
\begin{equation}
\label{eq_BFs}
\sum_{\lambda} f_\lambda e^{i  \lambda y} , \quad y\in \mR,
\end{equation}
where
$$
f_\lambda:= \lim_{N\rightarrow \infty} \frac{1}{2N}
\int_{[-N,N]}   e^{-i \lambda y} f(y)dy, \quad
\lambda \in \mR,
$$
and the sum in \eqref{eq_BFs} is taken over the set $
\sigma(f):=\{\lambda \in \mR\;|\; f_\lambda \neq 0\}$, called the {\em
  Bohr-Fourier spectrum} of $f$. The Bohr-Fourier spectrum of every
$f\in AP$ is at most a countable set.

We have $\widehat{L^1(\mR)}\cap AP=\{0\}$. Indeed such an almost periodic function must have limit zero at 
$\pm\infty$, and so it must be a constant equal to zero. This follows, for example, from the 
normality of the translates of almost periodic functions 
\cite[Chapter~I, Section~2, p.14]{Cor}, which says that if $f$ is an almost periodic function, then 
 any sequence of the form $(f(x+h_n))_{n\in \mN}$, where $h_n$ are real numbers, 
one can extract a subsequence converging uniformly on the real line. 

It can also be seen easily that $\widehat{L^1(\mR)}$ is an ideal in $\calA$, since if $f_a \in L^1(\mR)$ 
and $F_{AP}:=\sum_{k\in \mZ} f_k e^{-iyt_k}$ then 
$$
\sum_{k\in \mZ} f_k f_a(\cdot -t_k)
$$
is an absolutely convergent series in $L^{1}(\mR)$, 
whose Fourier transform is precisely $\widehat{f_a}\cdot F_{AP}$. 

If $R$ is a commutative unital ring, we denote by $\inv R$ the set of invertible elements of $R$. 

If $F =\widehat{f_a}+F_{AP} \in \inv \calA$, then we have for some $G=\widehat{g_a}+G_{AP} \in \calA$ that 
$$
(\widehat{f_a}+F_{AP})(\widehat{g_a}+ G_{AP})=\underbrace{\widehat{f_a} G+ F_{AP} \widehat{g_a}}_{\in \widehat{L^1(\mR)}} +F_{AP}G_{AP}=1.
$$ 
Using the fact that $\widehat{L^1(\mR)}$ is an ideal in $\calA$ and that $L^1(\mR) \cap AP=0$, we obtain $F_{AP} G_{AP}=1$, and so $F_{AP}(i \cdot) \in \inv AP$ 
(see \cite[Section~5.3]{SasN} for a different proof).
Also, again because  $\widehat{L^1(\mR)}$ is an ideal in $\calA$, we have that  $F_{AP}^{-1}\widehat{f_a}$ is the Fourier transform of a function in
$L^{1}(\mR)$, and so the map $y \mapsto
1+(F_{AP}(iy))^{-1}\widehat{f_a}(iy)=\frac{F(iy)}{F_{AP}(iy)}$ has a
well-defined winding number ${\tt w}$ around $0$; the definition is given below. Define $W: \inv
\calA \rightarrow \mR\times \mZ$ by
\begin{equation}   \label{Windex}
 W (F)= (w_{\textrm{av}}(F_{AP}), {\tt
  w}(1+F_{AP}^{-1} \widehat{f_a}) ),
\end{equation}
 where
$F=\widehat{f_a}+F_{AP} \in \inv \calA$, and
$$
\begin{array}{ll}
w_{\textrm{av}}(F_{AP})
:=
\displaystyle \lim_{R \rightarrow \infty} \frac{1}{2R}
\bigg( \arg \big(F_{AP}(iR)\big)-\arg\big(F_{AP}(-iR)\big)\bigg),
\\
{\tt w}(1+F_{AP}^{-1} \widehat{f_a})
:=
\displaystyle \frac{1}{2\pi}\bigg( \arg \big(1+(F_{AP}(iy)\big)^{-1} \widehat{f_a}(iy)  )
\bigg|_{y=-\infty}^{y=+\infty}\bigg).
\end{array}
$$
We also recall that  
$F=\widehat{f_a}+F_{AP} \in \calA$ is invertible if and only if for all $y\in
\mR$, $F(iy) \neq 0$ and $\displaystyle \inf_{y\in \mR} |F_{AP}(iy)| >0$ .

\begin{definition}
\label{def_nu_metric_calA_+}
For $P_1, P_2 \in \mS(\calA_+)$, with the normalized 
coprime factorizations
\begin{eqnarray*}
P_1&=& N_{1}/ D_{1},\\
P_2&=& N_{2}/ D_{2},
\end{eqnarray*}
we define
\begin{equation}
\label{eq_nu_metric}
d_{\calA_+} (P_1,P_2 ):=\left\{
\begin{array}{ll}
  \|\widetilde{G}_{2} G_{1}\|_{\infty} &
  \textrm{if } G_1^* G_2 \in \inv \calA \textrm{ and }\\
& \phantom{\textrm{if }}  W(G_1^* G_2)=(0,0), \\
  1 & \textrm{otherwise}, 
\end{array}\right.
\end{equation}
where the notation is as in Subsections~\ref{subsec1}-\ref{subsec6}. 
\end{definition}

It can be seen that this gives an extension of the classical Vinnicombe $\nu$-metric. 
Let $RH^\infty$ denote the set of all rational functions that are holomorphic and bounded in the open 
right half plane $\mC_{>0}:=\{s\in \mC: \textrm{Re}(s)> 0\}$. We use the notation $C(\mT)$ 
for the algebra of complex-valued continuous 
functions defined on the unit circle $\mT:=\{z\in \mC:|z|=1\}$, 
with all operations defined pointwise. 
If $f\in \inv C(\mT)$, then $f$ has a well-defined
(integral) winding number $w(f)\in \mZ$ with respect to $0$. 

Let $\varphi$ be the conformal map
$
\varphi:\mD \rightarrow \mC_{{\scriptscriptstyle >0}}
$ from the open unit disk $\mD$ to the open right half plane $\mC_{>0}$ given by
$$
\varphi(z)=\frac{1+z}{1-z} \quad (z\in \mD).
$$
Recall that the 
classical Vinnicombe $\nu$-metric is given as follows. For all 
$P_1,P_2$ in $\mS(RH^\infty)$, 
$$
d (P_1,P_2 )=\left\{
\begin{array}{ll}
  \|\widetilde{G}_{2} G_{1}\|_{\infty} &
  \textrm{if } ((G_1^* G_2)\circ \varphi) \in \inv C(\mT) \textrm{ and }\\
& \phantom{\textrm{if }}  w((G_1^* G_2)\circ \varphi)= 0 \\
  1 & \textrm{otherwise}, 
\end{array}\right.
$$
Clearly, if $ ((G_1^* G_2)\circ \varphi) \in \inv C(\mT)$, then the almost periodic part of $G_1^* G_2$ is a nonzero constant, and 
so the average winding number of $G_1^* G_2 $ must be zero and that 
$w((G_1^* G_2)\circ \varphi)={\tt w}((G_1^* G_2))$.  If 
$((G_1^* G_2)\circ \varphi) \not\in \inv C(\mT)$, then $G_1^* G_2
\not\in \inv \calA$, and so both $d(P_1,P_2)$ and $d_{\calA_+} (P_1,P_2 )$ 
are equal to $1$. Hence we have 
$$
d(P_1,P_2)=d_{\calA_+} (P_1,P_2 )
$$
whenever $P_1,P_2\in \mS(RH^\infty)$.

\subsection{The metric $d_{H^\infty}|_{\calA_+}$}

Let $H^\infty$ be the Hardy algebra, consisting of all bounded and
holomorphic functions defined on the open unit disk 
$$
\mD:= \{ z\in
\mC: |z| <1\}.
$$ 
Given $\rho\in (0,1)$, let ${\mA_{\rho}}$ be the open annulus
$$
{\mA_{\rho}}:=\{z\in \mC: \rho<|z|<1\}.
$$
We set $
C_{\textrm b}({\mA_{\rho}})=\{F:{\mA_{\rho}}\rightarrow \mC: f \textrm{ is continuous and bounded on }{\mA_{\rho}}\}.$

Let $\rho \in (0,1)$. With the norm defined by
$$
\|F\|_\infty :=\sup_{z\in {\mA_{\rho}}}|F(z)| \textrm{ for }
F\in C_{\textrm{b}}({\mA_{\rho}}),
$$
$C_{\textrm{b}}({\mA_{\rho}})$ is a unital semisimple commutative complex Banach
algebra with the involution $\cdot^\ast$ defined by 
$$
(F^\ast)(z)= \overline{F(z)}\quad (z\in {\mA_{\rho}},\;F\in C_{\textrm{b}}({\mA_{\rho}})).
$$
Let $\rho\in (0,1)$.  For
$f\in H^\infty$, define $\calI:H^\infty \rightarrow C_{\textrm{\em b}}({\mA_{\rho}})$ by
$$
(\calI (f))(z)=f(z) \quad (z\in {\mA_{\rho}},\; f \in H^\infty).
$$
Then $\calI$ is an injective map. Henceforth we will identify $H^\infty$ as a subset of
$C_{\textrm b}({\mA_{\rho}})$ via this map $\calI$.

We use the notation $C(\mT)$ for the Banach algebra of
complex-valued continuous functions defined on the unit circle
$\mT:=\{z\in \mC:|z|=1\}$, with all operations defined pointwise, 
with the supremum norm:
$$
\|f\|_\infty= \displaystyle \sup_{\zeta \in \mT}|f(\zeta)|\textrm{ for }f\in C(\mT),
$$
and the involution $\cdot^\ast$ defined pointwise:
$$
f^*(\zeta)=\overline{f(\zeta)}\quad (\zeta \in \mT).
$$
If $F\in \inv C_{\textrm b}({\mA_{\rho}})$, then for each
$r\in (\rho,1)$, the map $F_r:\mT \rightarrow \mC$, given by
$$
F_r(\zeta)=F(r\zeta) \quad (\zeta \in \mT),
$$
belongs to $ \inv C(\mT)$, and so each $F_r$ has a well-defined
(integral) winding number $w(F_r)\in \mZ$ with respect to $0$. By the
local constancy of the winding number $w: \inv C(\mT) \rightarrow \mZ$, 
 $r\mapsto w(F_r)$ is constant on $(\rho,1)$. That is, 
if $F\in \inv C_{\textrm{b}}({\mA_{\rho}})$, and $\rho<r<r'<1$, then
$$
w(F_r)=w(F_r').
$$
We now define the map ${\tt W}:\inv C_{\textrm b}({\mA_{\rho}})\rightarrow \mZ$ by setting
$$
{\tt W}(F)=w(F_r) \quad (r\in (\rho, 1), \;F\in \inv C_{\textrm b}({\mA_{\rho}})).
$$
Then ${\tt W}$ is well-defined.

As before, let $\varphi$ be the conformal map
$
\varphi:\mD \rightarrow \mC_{{\scriptscriptstyle >0}}
$ given by
$$
\varphi(z)=\frac{1+z}{1-z} \quad (z\in \mD).
$$
For $P_1, P_2 \in \mS(\calA_+)$, with the normalized 
coprime factorizations
\begin{eqnarray*}
P_1&=& N_{1}/ D_{1},\\
P_2&=& N_{2} / D_{2},
\end{eqnarray*}
we define
\begin{equation}
\label{eq_nu_metric_specialized}
d_{H^\infty}^\rho|_{\calA_+} (P_1,P_2 )=\!\left\{\!
\begin{array}{ll}
  \|(\widetilde{G}_{2} G_{1})\circ \varphi \|_{\infty} &
  \textrm{if } (G_1^* G_2)\circ \varphi \in \inv C_{\textrm{b}}({\mA_{\rho}}) \textrm{ and }\\
  &\phantom{\textrm{if }\;}{\tt W} ((G_1^* G_2)\circ \varphi) =0, \\
  1 & \textrm{otherwise},
\end{array} \right.
\end{equation}
where the notation is as in Subsections~\ref{subsec1}-\ref{subsec6}.

It follows from \cite{Sas} that $d_{H^\infty}$ defined by 
$$
d_{H^\infty}|_{\calA_+} (P_1,P_2 )=\lim_{\rho\rightarrow 1} d_{H^\infty}^\rho(P_1,P_2)
$$ 
actually defines a metric and if $P_1,P_2\in \mS(RH^\infty)$, then 
 $
d(P_1,P_2)= d_{H^\infty}|_{\calA_+} (P_1,P_2 )$. 
Thus this is also an extension of the classical Vinnicombe metric.

\section{An example of $P_1$, $P_2$ for which $d_{\calA_+}(P_1,P_2)\neq d_{ H^\infty}(P_1,P_2)$ }
\label{section_example}

Let $P$ be given by 
$$
P(s)=\frac{\alpha}{\beta}e^{-s}.
$$
where $\alpha, \beta$ are nonzero real numbers and $\alpha^2+\beta^2=1$. 
Set 
\begin{eqnarray*}
 N&:=& \alpha e^{-s},\\
 D&:=& \beta.
\end{eqnarray*}
Then 
$$
0\cdot N + \displaystyle \frac{1}{\beta} \cdot D=1,
$$ 
and so $N,D$  are coprime in $\calA_+$. Also,  
$$
N^* \cdot N+D^*\cdot  D=\alpha e^{-\overline{s}}\cdot \alpha e^{-s} +\beta^2=\alpha^2+\beta^2=1
$$
on $i\mR$. 
Thus 
$$
P=N / D
$$
is a normalized coprime factorization of $P$. 

Now choose a real number $r$ such that $\displaystyle \frac{1}{\sqrt{2}}<r<1$, and set 
\begin{eqnarray*} 
 P_1&=& \frac{r}{\sqrt{1-r^2}} e^{-s},\\
 P_2&=& \frac{r}{-\sqrt{1-r^2}} e^{-s}.
\end{eqnarray*}
Then we have 
\begin{eqnarray*}
G_1^* G_2&=& \left[ \begin{array}{cc} \overline{N_1}& \overline{D_1} \end{array} \right]
\left[ \begin{array}{cc} N_2 \\ D_2 \end{array} \right]
=\left[ \begin{array}{cc} re^{-\overline{s}} & \sqrt{1-r^2} \end{array} \right]
\left[ \begin{array}{cc} re^{-s} \\ -\sqrt{1-r^2} \end{array} \right]
\\
&=&r^{2} e^{-\overline{s}-s} -(1-r^2)=r^2e^{-2\textrm{Re}(s)}-(1-r^2).
\end{eqnarray*}
Thus $(G_1^* G_2)|_{i\mR}=r^2-(1-r^2)=2r^2-1>2\cdot \displaystyle \frac{1}{2}-1=0$. 
Hence $G_1^* G_2\in \inv \calA$ and $w_{\textrm{av}}(G_1^* G_2)=0$. Thus $W(G_1^* G_2)=(0,0)$. 

Also, 
\begin{eqnarray*}
\widetilde{G}_2 G_1&=& \left[ \begin{array}{cc} -D_2 & N_2\end{array} \right]
\left[ \begin{array}{cc} N_1 \\ D_1 \end{array} \right]
=\left[ \begin{array}{cc} \sqrt{1-r^2}  & re^{-s}  \end{array} \right]
\left[ \begin{array}{cc} re^{-s} \\ \sqrt{1-r^2} \end{array} \right]
\\
&=& 2 r \sqrt{1-r^2} e^{-s}.
\end{eqnarray*}
Hence by the Arithmetic Mean-Geometric Mean inequality, we have 
$$
\|\widetilde{G}_2 G_1\|_\infty = 2r \sqrt{1-r^2} <r^2+(1-r^2)=1,
$$
where we do have strict inequality since $r^2\neq 1-r^2$ (because $r\neq \displaystyle \frac{1}{\sqrt{2}}$). 
Consequently, 
$$
d_{\calA_+} (P_1,P_2)=\|\widetilde{G}_2 G_1\|_\infty<1.
$$

Next we will show that  $d_{H^\infty}(P_1,P_2)=1$. Note that if $\rho$ is in $(0,1)$, then the circle $\rho \mT$  is mapped under 
 the conformal map $\varphi: \mD \rightarrow \mC_{{\scriptscriptstyle >0}}
$, given by
$$
\varphi(z)=\frac{1+z}{1-z} \quad (z\in \mD),
$$
onto the circle 
$$
\frac{1+\rho^2}{1-\rho^2}+ \frac{2\rho}{1-\rho^2} \mT
$$
in the open right half plane. This circle intersects the real axis at the points $z_1<z_2$, where 
\begin{eqnarray*}
z_1&=&\frac{1+\rho^2}{1-\rho^2}-\frac{2\rho}{1-\rho^2} =\frac{1-\rho}{1+\rho},\\
z_2&=&\frac{1+\rho^2}{1-\rho^2}+\frac{2\rho}{1-\rho^2} =\frac{1+\rho}{1-\rho}.
\end{eqnarray*}
It is clear that for $\rho$ close enough to $1$, 
\begin{eqnarray*}
(G_1^* G_2)(z_1)&\approx &r^2\cdot 1-(1-r^2)=2r^2-1>2\cdot \frac{1}{2}-1=0,\\
(G_1^* G_2)(z_2)&\approx&r^2\cdot 0-(1-r^2) =-(1-r^2)<0.
\end{eqnarray*}
But $G_1^* G_2=r^2e^{-2\textrm{Re}(s)}-(1-r^2)$ is always real-valued. By the Intermediate Value Theorem, 
it follows that it must be a zero somewhere in $\mA_\rho$ and $G_1^* G_2$ can't belong to $\inv C_{\textrm b}({\mA_{\rho}})$. 
In fact, all zeros  belong  to an arc of a circle with center on the real axis, tangent to $z=1$, as can be seen easily in 
the right half plane. 
Hence $d_{H^\infty}(P_1,P_2)=1$. 

\section{Equivalence of $d_{\calA_+}$ and  $d_{H^\infty}|_{\calA_+}$} 
\label{section_equivalence}

\subsection{An alternative expression for $d_{H^\infty}|_{\calA_+}$} 

We begin by giving an alternative expression for $d_{H^\infty}$. 

If $M \in \mC^{p\times m}$, then the set of nonzero  eigenvalues of $MM^*$
  and $M^* M$ coincide. We denote by $\overline{\sigma}(M)$ the square root of the largest eigenvalue of $M^*M$ (or equivalently $MM^*$). 
 For a matrix $M \in \calA^{p\times m}$, we set
\begin{equation}
\label{norm}
\|M\|_{\infty}= \sup_{y\in \mR} \overline{\sigma}(M(iy)).
\end{equation}

\begin{proposition}
\label{prop_alt_exp_nu_metric}
If $P_1, P_2\in \mS(\calA_+)$, then for each $\rho\in (0,1)$, 
 $$
d_{H^\infty}^\rho (P_1, P_2)= \displaystyle \inf_{\substack{Q \in \textrm{\em inv } C_{\textrm{\em b}}({\mA_{\rho}}) ,\\
    {\tt W}( Q)=0}} \|G_1 -G_2 Q\|_\infty.
$$
\end{proposition}
\begin{proof} Let $Q \in \inv  C_{\textrm{b}}({\mA_{\rho}})$ and
${\tt W}(Q)=0$. We have
\begin{eqnarray*}
\|G_1 -G_2 Q\|_\infty &=&
\left\| \left[ \begin{array}{cc} G_2^* \\
\widetilde{G}_2 \end{array} \right] (G_1-G_2 Q) \right\|_\infty
\;\; \textrm{(as }
\left[\begin{array}{cc} G_2 & \widetilde{G}_2^* \end{array}\right]
\left[\begin{array}{c} G_2^* \\\widetilde{G}_2 \end{array}\right]=I\textrm{)} \\
&=&
\left\| \left[\begin{array}{c} G_2^* G_1-Q\\ \widetilde{G}_2G_1 \end{array}\right]  \right\|_\infty
\quad \textrm{(since } \widetilde{G}_2 G_2=0 \textrm{ and }  G_2^*G_2=I\textrm{)} \\
&\geq & \| \widetilde{G}_2G_1 \|_\infty .\phantom{\left[\begin{array}{c} A \\B \end{array}\right]}
\end{eqnarray*}
So if $G_2^* G_1 \in \inv C_{\textrm{b}}({\mA_{\rho}})$ and ${\tt W}(G_2^* G_1)=0$, then from
the above it follows that $ \|G_1 -G_2 Q\|_\infty \geq \|
\widetilde{G}_2G_1 \|_\infty=d_{H^\infty}^\rho(P_1,P_2)$.  As the choice of $Q$
above was arbitrary, we obtain
\begin{equation}
\label{eq_nu_metrix_alt_exp_ineq_1}
\inf_{\substack{Q \in \inv  C_{\textrm{b}}({\mA_{\rho}}),\\
{\tt W}( Q)=0}} \|G_1 -G_2 Q\|_\infty \geq d_{H^\infty}^\rho(P_1,P_2).
\end{equation}
If we define $Q_0:= G_2^* G_1 \in C_{\textrm{b}}({\mA_{\rho}})$,
then $Q_0\in \inv C_{\textrm{b}}({\mA_{\rho}})$ and ${\tt W} (Q_0)=0$, and so
\begin{eqnarray*}
  \inf_{\substack{Q \in \inv C_{\textrm{b}}({\mA_{\rho}}) ,\\
      {\tt W}(Q)=0}} \|G_1 -G_2 Q\|_\infty
  &\leq&  \|G_1-G_2 Q_0\|_\infty
  =
  \left\| \left[\begin{array}{c} G_2^* G_1-Q_0\\ \widetilde{G}_2G_1 \end{array}\right]  \right\|_\infty
  \\
  &=&\left\| \left[\begin{array}{c} 0\\ \widetilde{G}_2G_1 \end{array}\right]  \right\|_\infty
  =
  \|\widetilde{G}_2G_1\|_\infty =d_{H^\infty}^\rho(P_1,P_2).
\end{eqnarray*}
From this and \eqref{eq_nu_metrix_alt_exp_ineq_1}, the claim in the
proposition follows for the case when $G_2^* G_1\in \inv C_{\textrm{b}}({\mA_{\rho}})$ and ${\tt W}
(G_2^* G_1)=0$.

Now let $Q\in \inv C_{\textrm{b}}({\mA_{\rho}})$ be such that ${\tt W} (Q)=0$ and
$\|G_1-G_2 Q\|_\infty <1$. Using $G_1^* G_1=1$, we see that $\|G_1^*\|_\infty =1$ and 
$$
\|1-G_1^* G_2 Q\|_\infty =\|G_1^*(G_1-G_2 Q)\|_\infty
\leq\|G_1^*\|_\infty \|G_1-G_2 Q\|_\infty<1\cdot 1=1.
$$
So $ G_1^* G_2 Q=1-(1-G_1^* G_2 Q) $ is invertible as an element of
$C_{\textrm{b}}({\mA_{\rho}})$. Consider the map $H:[0,1]\rightarrow \inv C_{\textrm{b}}({\mA_{\rho}}) $ given by $
H(t)= 1-t(1-G_1^* G_2 Q)$, $t\in [0,1]$.  By the homotopic invariance
of the index ${\tt W}$ \cite[Proposition~2.1]{BalSas},
$$
0={\tt W}(1)={\tt W} (H(0))={\tt W} (H(1))={\tt W}(G_1^* G_2 Q).
$$
As ${\tt W}(Q)=0$, we obtain that ${\tt W}(G_1^* G_2)=0$. So we have shown
that if there is a $Q\in C_{\textrm{b}}({\mA_{\rho}})$ such that $Q\in \inv C_{\textrm{b}}({\mA_{\rho}})$, ${\tt W}
(Q)=0$ and $\|G_1-G_2 Q\|_\infty <1$, then $G_1^* G_2 \in 
\inv C_{\textrm{b}}({\mA_{\rho}})$ and ${\tt W}(G_1^* G_2)=(0,0)$. Thus we have that if either $G_1^* G_2
\not\in \inv C_{\textrm{b}}({\mA_{\rho}})$ or $G_1^* G_2 \in \inv C_{\textrm{b}}({\mA_{\rho}})$ but ${\tt W}(G_1^*
G_2)\neq 0$, then for all elements $Q\in C_{\textrm{b}}({\mA_{\rho}})$ such that $Q\in \inv
 C_{\textrm{b}}({\mA_{\rho}})$, ${\tt W} (Q)=0$, we have that $\|G_1-G_2 Q\|_\infty \geq 1$, and
so
$$
\inf_{\substack{Q \in \inv C_{\textrm{b}}({\mA_{\rho}}) ,\\
{\tt W}(Q)=0}} \|G_1 -G_2 Q\|_\infty \geq 1=d_{H^\infty}^\rho(P_1,P_2).
$$
Also, with $Q_n:=\displaystyle \frac{1}{n} I$ ($n\in \mN$), $Q_n\in \inv C_{\textrm{b}}({\mA_{\rho}})$ and
${\tt W} ( Q_n)=0$. We have
$$
\|G_1-G_2 Q_n\|_\infty \leq \|G_1\|_\infty + \|G_2\|_\infty
\|Q_n\|_\infty\leq 1+1\cdot \frac{1}{n}.
$$
Hence
\begin{eqnarray*}
\inf_{\substack{Q \in \inv C_{\textrm{b}}({\mA_{\rho}}) ,\\
    {\tt W}(Q)=0}} \|G_1 -G_2 Q\|_\infty &\leq& \inf_{n\in \mN} \|G_1 -G_2
 Q_n\|_\infty\\
&\leq& \inf_{n\in \mN} \left(1+\frac{1}{n}\right)=1=d_{H^\infty}^\rho(P_1,P_2).
\end{eqnarray*}
Consequently,
 $
\displaystyle \inf_{\substack{Q \in \inv C_{\textrm{b}}({\mA_{\rho}}) ,\\
{\tt W}(Q)=0}} \|G_1 -G_2 Q\|_\infty=1=d_{H^\infty}^\rho(P_1,P_2)$.
\end{proof}

For $P\in \mS(\calA_+)$, set $\mu_{\textrm{opt},\calA_+}(P):=\displaystyle \sup_{C\in \mS(\calA_+)} \mu_{P,C}$, where 
$$
\mu_{P,C}=\left\{ \begin{array}{ll}
\|H(P,C)\|_{\infty}^{-1} &\textrm{if }P \textrm{ is stabilized by }C,\\
0 & \textrm{otherwise,}
\end{array}\right.
$$
and 
$$
H(P,C):= \left[\begin{array}{cc} P \\ 1 \end{array} \right]
(1-CP)^{-1} \left[\begin{array}{cc} -C & 1 \end{array} \right].
$$
We remark that first of all $\mu_{\textrm{opt},\calA_+}(P)>0$ because every 
$P \in \mS(\calA_+)$ has a coprime factorisation, and we know that the coprime factorization 
 gives a stabilizing controller. Secondly, as $\mu_{P,C}$ is always bounded above by $1$ (see \cite[Remark~4.3]{BalSas}), we have that 
$\mu_{\textrm{opt},\calA_+}(P)\leq 1$.

\subsection{The gap-metric} In this subsection we will recall the gap-metric topology for unstable
plants over the ring $\calA_+$. We will also recall a few known results from \cite{Sas0}
lemmas which will be used in the next subsection in order to prove our claimed equivalence.

\begin{definition}[Graph of a system]
  For $P\in \mS(\calA_+)$, with the normalized coprime factorization $
  P= N / D$, we define the {\em graph of} $P$, denoted by $\calG$,
  to be the following subspace of the Hardy space
  $H^2\times H^2$:\label{pageref_calG}
$$
\calG=G H^2=\left\{\left[ \begin{array}{cc} N\varphi \\ D
      \varphi\end{array} \right]: \varphi \in H^2\right\}.
$$
\end{definition}

Here $H^2$ denotes the Hardy space of all holomorphic functions defined in the open right half plane $\mC_{>0}:=\{s\in \mC: \textrm{Re}(s)>0\}$ such that 
$$
\sup_{\zeta>0}\|f(\zeta+i\cdot)\|_{L^2(\mR)}<+\infty.
$$

One can see that $\calG$ is a closed subspace of $H^2\times
  H^2$. Suppose that
$$
\left[ \begin{array}{cc} N\\ D
      \end{array} \right]\varphi_n 
\stackrel{n\rightarrow\infty}{\longrightarrow} 
\left[ \begin{array}{cc} f\\ g
      \end{array} \right]
$$ 
in $H^2\times H^2$. If $X,Y\in \calA_+$ are such that $XN+YD=1$,
then using the fact that elements from $\calA_+$ are bounded and
holomorphic in the right half plane, we obtain that
$$
\varphi_n \stackrel{n\rightarrow\infty}{\longrightarrow} Xf+Yg =:\varphi
$$
in $H^2$. Consequently, using the fact that $N,D$ are bounded and holomorphic in the open right half plane, we obtain 
$$ 
\left[ \begin{array}{cc} N\\ D
      \end{array} \right]\varphi_n 
\stackrel{n\rightarrow\infty}{\longrightarrow}
\left[ \begin{array}{cc} N\\ D
      \end{array} \right]\varphi\in \calG.
$$

We denote the orthogonal projection from $H^2\times
H^2$ onto $\calG$ by $P_{\calG}$.\label{pageref_P_calG}

\begin{definition}[Gap-metric $d_g$]
\label{def_graph_metric}
For $P_1, P_2 \in \mS(\calA_+)$, with the normalized coprime
factorizations $P_1= N_{1}/ D_{1}$ and $P_2= N_{2}/ D_{2}$, we
define \label{pageref_d_g}
\begin{equation}
\label{eq_graph_metric}
d_{\textrm{g}} (P_1,P_2 ):=
\|P_{\calG_1}-P_{\calG_2}\|_{\calL(H^2\times H^2)}.
\end{equation}
\end{definition}

We recall \cite[Proposition~4.9~and~Theorem~1.1]{Sas0}:

\begin{proposition}
\label{prop_alt_exp_gap_metric}
If $P_1,P_2\in \mS(\calA_+)$, then
$$
d_{\textrm{\em g}}(P_1,P_2)=\inf_{Q\in \textrm{\em inv } \calA_+} \|G_1-G_2 Q\|_\infty.
$$
\end{proposition}

\begin{proposition}
For $P_1,P_2\in \mS(\calA_+)$:
\begin{equation}
\label{eq_equiv_inqs}
d_{\textrm{\em g}}(P_1,P_2)\mu_{\textrm{\em{opt}},\calA_+}(P_1)\leq d_{\calA_+}(P_1,P_2) \leq d_{\textrm{\em g}}(P_1,P_2).
\end{equation}
\end{proposition}

\subsection{Equivalence}
Let $P_1,P_2\in \mS(\calA_+)$. Then
\begin{eqnarray*}
 d_{H^\infty}^\rho(P_1,P_2)&=& \displaystyle \inf_{\substack{Q \in \textrm{inv } C_{\textrm{b}}({\mA_{\rho}}) ,\\
    {\tt W}( Q)=0}} \|G_1 -G_2 Q\|_\infty\\
&\leq & \displaystyle \inf_{\substack{Q \in (\textrm{inv } C_{\textrm{b}}({\mA_{\rho}}) )\cap H^\infty ,\\
    {\tt W}( Q)=0}} \|G_1 -G_2 Q\|_\infty\\
&=& \displaystyle \inf_{Q \in \textrm{inv }  H^\infty}  \|G_1 -G_2 Q\|_\infty\\
&\leq & \displaystyle \inf_{Q \in \textrm{inv }  \calA_+}  \|G_1 -G_2 Q\|_\infty\\
&=& d_{\textrm{\em g}}(P_1,P_2)\\
&\leq& \frac{d_{\calA_+}(P_1,P_2)}{ \mu_{\textrm{opt},\calA_+}(P_1) }.
\end{eqnarray*}
Consequently,
\begin{equation}
 \label{equiv_1_ineq}
d_{H^\infty}(P_1,P_2)=\lim_{\rho\rightarrow 1}  d_{H^\infty}^\rho(P_1,P_2) \leq \frac{d_{\calA_+}(P_1,P_2)}{ \mu_{\textrm{opt},\calA_+}(P_1) }.
\end{equation}

Next we will show that 
$$
d_{\calA_+}(P_1,P_2) \mu_{\textrm{opt},\calA_+}(P_1)  \leq d_{H^\infty}(P_1,P_2).
$$
This inequality is trivially satisfied if $d_{H^\infty}(P_1,P_2) \geq \mu_{\textrm{opt},\calA_+}(P_1) $, since we know that 
$d_{\calA_+}(P_1,P_2) \leq 1$. 

So we will only consider the case when $d_{H^\infty}(P_1,P_2)<\mu_{\textrm{opt},\calA_+}(P_1) $. 
In particular, $\mu_{\textrm{opt},\calA_+}(P_1) >0$. This inequality implies that there is an element 
 $C_0\in \mS(\calA_+)$ that stabilizes $P_1$. Moreover, 
$d_{H^\infty}(P_1,P_2)<\mu_{P_1,C_0}$. Using the fact \cite[Theorem~3.15]{Sas} that 
$$
\mu_{P_2,C_0}\geq \mu_{P_1,C_0}-d_{H^\infty}(P_1,P_2),
$$
it follows that $C_0$ stabilizes (in $H^\infty$) $P_2$ as well. But by the corona theorems for $H^\infty$ and for $\calA_+$  it follows that 
$C_0$ stabilizes $P_2$ in $\calA_+$ too.

Define $ Q_0:= ( \widetilde{K}_0
G_1)^{-1}\widetilde{K}_0 G_2$. By \cite[Proposition~4.4]{BalSas}, we know that  $\widetilde{K}_0 G_2$ is invertible in $\calA_+$. We have 
$$
G_2- G_1 Q_0 = G_2- G_1 (
\widetilde{K}_0 G_1)^{-1}\widetilde{K}_0G_2 = (I-G_1 ( \widetilde{K}_0
G_1)^{-1}\widetilde{K}_0)G_2 .
$$ 
Also
$$
I-\left[\begin{array}{cc} P_1 \\ 1 \end{array}\right] (1-C_0 P_1)^{-1} \left[\begin{array}{cc}-C_0 & 1 \end{array}\right]
 =
\left[\begin{array}{cc} 1 \\ C_0 \end{array}\right] (1-P_1C_0)^{-1}
 \left[\begin{array}{cc}1 & -P_1 \end{array}\right].
$$
that is, $ I-G_1 ( \widetilde{K}_0 G_1)^{-1}\widetilde{K}_0=K_0
(\widetilde{G}_1 K_0)^{-1}\widetilde{G}_1$.  Thus
$$
G_2- G_1 Q_0 =K_0 (\widetilde{G}_1 K_0)^{-1}\widetilde{G}_1G_2.
$$
Then we use $\|K_0\|\leq 1$ (since $K_0^* K_0=1$) to obtain
\begin{eqnarray*}
\|G_2 -G_1 Q_0\|_\infty &=& \|K_0 (\widetilde{G}_1 K_0)^{-1}\widetilde{G}_1   G_2\|_\infty\\
&\leq & \|K_0\|_\infty\|(\widetilde{G}_1 K_0)^{-1}\widetilde{G}_1   G_2\|_\infty\\
&\leq & 1 \cdot \|(\widetilde{G}_1 K_0)^{-1}\widetilde{G}_1   G_2\|_\infty\\
&\leq &  \|(\widetilde{G}_1 K_0)^{-1}\|_\infty\|\widetilde{G}_1   G_2\|_\infty.
\end{eqnarray*}
As for each $C$, $\mu_{P_1,C}\leq 1$, we have
$\mu_{\textrm{opt},\calA_+}(P_1)\leq 1$. So 
$$
d_{H^\infty}(P_1,P_2)<\mu_{\textrm{opt},\calA_+}(P_1)\leq 1,
$$ 
and we obtain $d_{H^\infty}(P_1,P_2)=\|\widetilde{G}_1 G_2\|_\infty$.

From \cite[Propositions~4.2,4.5]{BalSas}, $ \|(\widetilde{G}_1
K_0)^{-1}\|_\infty=1/\mu_{C_0,P_1}=1/\mu_{P_1,C_0}$.  So
$$
\|G_2 -G_1 Q_0\|_\infty\leq  \| (\widetilde{G}_1 K_0)^{-1}\|_\infty\|\widetilde{G}_1   G_2\|_\infty
\leq \frac{d_{H^\infty}(P_1,P_2)}{\mu_{P_1,C_0}}.
$$
Thus
$$
d_{\textrm{g}}(P_1,P_2) = \displaystyle \inf_{Q\in \inv \calA_+} \|G_1-G_2 Q\|_\infty \leq
\|G_1-G_2 Q_0\|\leq d_{H^\infty}(P_1,P_2)/\mu_{P_1,C_0}.
$$
But 
$$
d_{\textrm{g}}(P_1,P_2)\geq d_{\calA_+}(P_1,P_2).
$$
Hence 
$$
\mu_{P_1,C_0}\cdot d_{\calA_+}(P_1,P_2)\leq d_{H^\infty}(P_1,P_2).
$$
As this inequality holds for any $C_0$ that stabilizes $P_1$ (in $\calA_+$)  for which
there holds $d_{H^\infty}(P_1,P_2)<\mu_{P_1,C_0}$,  we can choose a sequence
$(C_{0,n})_{n\in \mN}$ such $\mu_{P_1,C_{0,n}} \rightarrow
\mu_{\textrm{opt},\calA_+}(P_1)$ as $n\rightarrow \infty$. Thus 
\begin{equation}
 \label{equiv_2_ineq}
\mu_{\textrm{opt},\calA_+}(P_1) \cdot d_{\calA_+}(P_1,P_2)
\leq d_{H^\infty}(P_1,P_2).
\end{equation}

Finally, from \eqref{equiv_1_ineq} and \eqref{equiv_2_ineq}, we have 
$$
\mu_{\textrm{opt},\calA_+}(P_1) \cdot d_{\calA_+}(P_1,P_2)
\leq d_{H^\infty}(P_1,P_2) \leq \frac{d_{\calA_+}(P_1,P_2)}{ \mu_{\textrm{opt},\calA_+}(P_1) }.
$$

% \begin{remark}
% \label{Adams_remark}
%  The above inequalities in fact give more than the mere equivalence of the metrics 
%  $d_{\calA_+}$ and $d_{H^\infty}|_{\calA_+}$ on stabilizable plants over $\calA_+$. 
%  
% Equivalence of metrics $d_1,d_2$ on $X$ means that for each $x\in X$ there exist functions 
% $r^+_x ,r^-_x : (0,\infty)\rightarrow (0,\infty)$ such that for all $y\in X$, we have 
% $$
% r_x^- (d_1(x,y)) \leq d_2(x,y) \leq r_x^+ (d_1(x,y)) \quad (x\in X).
% $$
% It is usual to call metrics {\em strongly equivalent} if $r^+_x$ and $r^-_x$ are independent of $x$ and 
% are linear. (One can then always find a constant $k>0$ such that  for all $s\in (0,\infty)$, $ r^+(s)=ks$ and $ r^-(s)=\frac{1}{k}s$.) 
% We prefer to use the terminology {\em globally linearly equivalent} for such a pair of metrics. 
% 
% In our case, we have the existence of a function $k:X\rightarrow (0,\infty)$ such that  for all $s\in (0,\infty)$, 
%  $$
%  r^+_x(s)= k(x)s,\;\; r^-_x(s)=\frac{1}{k(x)}s \quad (x\in X)
%  $$
% and we then call $d_1,d_2$ {\em poinwtise linearly equivalent}.
% 
% We have shown that $d_{\calA_+}$ and $d_{H^\infty}|_{\calA_+}$ are pointwise linearly equivalent. 
% \end{remark}

\begin{remark}
We also mention that in this article we have only considered {\em
  single} input and {\em single} output control systems.  However, the
metrics $d_{\calA^+}$, $d_{H^\infty}$ can also be defined on plants
with multiple inputs and/or outputs as well; see \cite{BalSas} and
\cite{Sas}. One can ask if the induced topologies (on such matricial
stabilizable plants over $\calA_+$) are still equivalent.  We leave
this as an open problem. 

Our route of proving the equivalence in the case of single input
single output systems in this article is by appealing to the results
from \cite{Sas0}, which unfortunately are also available in only the
scalar case. Whether the matricial analogue of the result from
\cite{Sas0} holds is also open. If that result were available, then 
the same proof in this article, mutatis mutandis, would also 
yield the extension of the result in this article to the matricial case.
\end{remark}

% \medskip 
% 
% \noindent {\bf Acknowledgement:} The second author thanks Professor Adam Ostaszewski for Remark~\ref{Adams_remark}.

\end{document}